\documentclass[12pt]{amsart}
\usepackage{amsmath,amssymb,amsthm,amsfonts,amscd,amsopn}
\usepackage{hyperref}
\usepackage[usenames]{color}
\usepackage{eucal, mathrsfs}
\usepackage{amsmath,amssymb,amsthm,%
amsfonts,
amscd,amsopn}
\usepackage{hyperref}
\usepackage{xspace}
\usepackage{
color}
\usepackage[all]{xy}\xyoption{dvips}
\usepackage{comment} 
\usepackage{paralist}
\usepackage{stmaryrd}
\usepackage{enumitem}
\usepackage{supertabular}
\makeatletter

\newcounter{are-there-sections}
\def\me{S\'ANDOR J KOV\'ACS\xspace}

\def\mythanks{Supported in part by NSF Grant 
 DMS-0856185, and the
Craig McKibben and Sarah Merner Endowed Professorship in Mathematics at the
University of Washington.}

\def\myaddress{University of Washington, Department of Mathematics, Box 354350,
Seattle, WA 98195-4350, USA}

\def\myemail{skovacs@uw.edu\xspace}

\def\myurladdr{http://www.math.washington.edu/$\sim$kovacs\xspace}

\DeclareMathAlphabet{\smallchanc}{OT1}{pzc}%
                                 {m}{it}
\DeclareFontFamily{OT1}{pzc}{}
\DeclareFontShape{OT1}{pzc}{m}{it}%
             {<-> s * [1.100] pzcmi7t}{}
\DeclareMathAlphabet{\mathchanc}{OT1}{pzc}%
                                 {m}{it}

\newcommand{\mcL}{\mathchanc{L}}

\newcommand{\mcR}{\mathchanc{R}}

\DeclareFontFamily{OMS}{rsfs}{\skewchar\font'60}
\DeclareFontShape{OMS}{rsfs}{m}{n}{<-5>rsfs5 <5-7>rsfs7 <7->rsfs10 }{}
\DeclareSymbolFont{rsfs}{OMS}{rsfs}{m}{n}
\DeclareSymbolFontAlphabet{\scr}{rsfs}

\newcommand{\sD}{\scr{D}}

\newcommand{\sF}{\scr{F}}
\newcommand{\sG}{\scr{G}}
\newcommand{\sH}{\scr{H}}
\newcommand{\sI}{\scr{I}}

\newcommand{\sO}{\scr{O}}

\newcommand{\sS}{\scr{S}}

\newcommand{\sfA}{{\sf A}}
\newcommand{\sfB}{{\sf B}}

\newcommand{\bC}{\mathbb{C}}

\newcommand{\bH}{\mathbb{H}}

\newcommand{\bQ}{\mathbb{Q}}

\newcommand{\ol}{\overline}

\newcommand{\into}{\hookrightarrow}

\newcommand{\onto}{\twoheadrightarrow}
\newcommand{\isom}{\overset{\simeq\ }\longrightarrow}

\newcommand{\leteq}{\colon\!\!\!=}
\newcommand{\col}{\colon}

\newcommand{\sch}{\ensuremath{{\sf{Sch}}}\xspace}

\newcommand{\rpforward}[1]{\myR{#1}_*}

\DeclareMathOperator{\exc}{Exc}

\DeclareMathOperator{\id}{{id}}

\DeclareMathOperator{\Ob}{{Ob}}

\DeclareMathOperator{\ob}{{Ob}}

\DeclareMathOperator{\Sing}{{Sing}}

\DeclareMathOperator{\supp}{{supp}}

\newcommand{\factor}[2]{\left. \raise 2pt\hbox{\ensuremath{#1}} \right/
        \hskip -2pt\raise -2pt\hbox{\ensuremath{#2}}}

\newcommand{\disjoint}{\overset{_\kdot}\cup}
\newcommand{\myR}{{\mcR\!}}

\newcommand{\myL}{{\mcL}}

\newcommand{\blank}{\underline{\hskip 10pt}}

\newcommand{\kdot}{{{\,\begin{picture}(1,1)(-1,-2)\circle*{2}\end{picture}\ }}}
\newcommand{\mydot}{\kdot}

\newcommand{\cx}{\sf}
\newcommand{\DuBois}[1]{{\underline \Omega {}^0_{#1}}}

\newcommand{\Om}{\underline{\Omega}}

\def\coh#1.#2.#3.{H^{#1}(#2,#3)}
\def\dimcoh#1.#2.#3.{h^{#1}(#2,#3)}
\def\hypcoh#1.#2.#3.{\mathbb H_{\vphantom{l}}^{#1}(#2,#3)}
\def\loccoh#1.#2.#3.#4.{H^{#1}_{#2}(#3,#4)}
\def\dimloccoh#1.#2.#3.#4.{h^{#1}_{#2}(#3,#4)}
\def\lochypcoh#1.#2.#3.#4.{\mathbb H^{#1}_{#2}(#3,#4)}
\def\ses#1.#2.#3.{0  \longrightarrow  #1   \longrightarrow 
 #2 \longrightarrow #3 \longrightarrow 0} 
\def\sesshort#1.#2.#3.{0
 \rightarrow #1 \rightarrow #2 \rightarrow #3 \rightarrow 0}
\def\dist#1.#2.#3.{  #1   \longrightarrow 
 #2 \longrightarrow #3 \stackrel{+1}{\longrightarrow} } 
\def\CDdist#1.#2.#3.{  #1   @>>>  #2  @>>>   #3 @>+1>> }  
\def\shortses#1.#2.#3.{0  \rightarrow  #1   \rightarrow 
 #2  \rightarrow   #3 \rightarrow  0}
\def\shortdist#1.#2.#3.{  #1   \rightarrow 
 #2  \rightarrow   #3 \stackrel{+1}{\rightarrow} }  
\def\ddist#1.#2.#3.#4.#5.#6.{\CD
#1 @>>> #2 @>>> #3 @>+1>> \\
@VVV @VVV @VVV \\
#4 @>>> #5 @>>> #6 @>+1>> 
\endCD}
\def\ddistun#1.#2.#3.#4.#5.#6.{\CD
#1 @>>> #2 @>>> #3 @>+1>> \\
@. @VVV @VVV  \\
#4 @>>> #5 @>>> #6 @>+1>> 
\endCD}
\def\Iff#1#2#3{
\hfil\hbox{\hsize =#1
\vtop{\noin #2}
\hskip.5cm 
\lower.5\baselineskip\hbox{$\Leftrightarrow$}\hskip.5cm
\vtop{\noin #3}}\hfil\medskip}
\newcommand{\union}\cup
\newcommand{\intersect}\cap
\newcommand{\Union}\bigcup
\newcommand{\Intersect}\bigcap
\def\myoplus#1.#2.{\underset #1 \to {\overset #2 \to \oplus}}

\def\qis{\,{\simeq}_{\text{qis}}\,}

\begin{document}
\makeatletter
\newenvironment{refmr}{}{}
\renewcommand{\labelenumi}{{\rm (\thethm.\arabic{enumi})}}
%
\setitemize[1]{leftmargin=*,parsep=0em,itemsep=0.125em,topsep=0.125em}
\newcommand\james{M\raise .575ex \hbox{\text{c}}Kernan}

\renewcommand\thesubsection{\thesection.\Alph{subsection}}
\renewcommand\subsection{
  \renewcommand{\sfdefault}{pag}
  \@startsection{subsection}%
  {2}{0pt}{-\baselineskip}{.2\baselineskip}{\raggedright
    \sffamily\itshape\small
  }}
\renewcommand\section{
  \renewcommand{\sfdefault}{phv}
  \@startsection{section} %
  {1}{0pt}{\baselineskip}{.2\baselineskip}{\centering
    \sffamily
    \scshape
}}
\newcounter{lastyear}\setcounter{lastyear}{\the\year}
\addtocounter{lastyear}{-1}
\newcommand\sideremark[1]{%
\normalmarginpar
\marginpar
[
\hskip .45in
\begin{minipage}{.75in}
\tiny #1
\end{minipage}
]
{
\hskip -.075in
\begin{minipage}{.75in}
\tiny #1
\end{minipage}
}}
\newcommand\rsideremark[1]{
\reversemarginpar
\marginpar
[
\hskip .45in
\begin{minipage}{.75in}
\tiny #1
\end{minipage}
]
{
\hskip -.075in
\begin{minipage}{.75in}
\tiny #1
\end{minipage}
}}
\newcommand\Index[1]{{#1}\index{#1}}
\newcommand\inddef[1]{\emph{#1}\index{#1}}
\newcommand\noin{\noindent}
\newcommand\hugeskip{\bigskip\bigskip\bigskip}
\newcommand\smc{\sc}
\newcommand\dsize{\displaystyle}
\newcommand\sh{\subheading}
\newcommand\nl{\newline}
\newcommand\get{\input /home/kovacs/tex/latex/} 
\newcommand\Get{\Input /home/kovacs/tex/latex/} 
\newcommand\toappear{\rm (to appear)}
\newcommand\mycite[1]{[#1]}
\newcommand\myref[1]{(\ref{#1})}
\newcommand\myli{\hfill\newline\smallskip\noindent{$\bullet$}\quad}
\newcommand\vol[1]{{\bf #1}\ } 
\newcommand\yr[1]{\rm (#1)\ } 
\newcommand\cf{cf.\ \cite}
\newcommand\mycf{cf.\ \mycite}
\newcommand\te{there exist}
\newcommand\st{such that}
\newcommand\myskip{3pt}
\newtheoremstyle{bozont}{3pt}{3pt}%
     {\itshape}
     {}
     {\bfseries}
     {.}
     {.5em}
     {\thmname{#1}\thmnumber{ #2}\thmnote{ \rm #3}}
\newtheoremstyle{bozont-sf}{3pt}{3pt}%
     {\itshape}
     {}
     {\sffamily}
     {.}
     {.5em}
     {\thmname{#1}\thmnumber{ #2}\thmnote{ \rm #3}}
\newtheoremstyle{bozont-sc}{3pt}{3pt}%
     {\itshape}
     {}
     {\scshape}
     {.}
     {.5em}
     {\thmname{#1}\thmnumber{ #2}\thmnote{ \rm #3}}
\newtheoremstyle{bozont-remark}{3pt}{3pt}%
     {}
     {}
     {\scshape}
     {.}
     {.5em}
     {\thmname{#1}\thmnumber{ #2}\thmnote{ \rm #3}}
\newtheoremstyle{bozont-def}{3pt}{3pt}%
     {}
     {}
     {\bfseries}
     {.}
     {.5em}
     {\thmname{#1}\thmnumber{ #2}\thmnote{ \rm #3}}
\newtheoremstyle{bozont-reverse}{3pt}{3pt}%
     {\itshape}
     {}
     {\bfseries}
     {.}
     {.5em}
     {\thmnumber{#2.}\thmname{ #1}\thmnote{ \rm #3}}
\newtheoremstyle{bozont-reverse-sc}{3pt}{3pt}%
     {\itshape}
     {}
     {\scshape}
     {.}
     {.5em}
     {\thmnumber{#2.}\thmname{ #1}\thmnote{ \rm #3}}
\newtheoremstyle{bozont-reverse-sf}{3pt}{3pt}%
     {\itshape}
     {}
     {\sffamily}
     {.}
     {.5em}
     {\thmnumber{#2.}\thmname{ #1}\thmnote{ \rm #3}}
\newtheoremstyle{bozont-remark-reverse}{3pt}{3pt}%
     {}
     {}
     {\sc}
     {.}
     {.5em}
     {\thmnumber{#2.}\thmname{ #1}\thmnote{ \rm #3}}
\newtheoremstyle{bozont-def-reverse}{3pt}{3pt}%
     {}
     {}
     {\bfseries}
     {.}
     {.5em}
     {\thmnumber{#2.}\thmname{ #1}\thmnote{ \rm #3}}
\newtheoremstyle{bozont-def-newnum-reverse}{3pt}{3pt}%
     {}
     {}
     {\bfseries}
     {}
     {.5em}
     {\thmnumber{#2.}\thmname{ #1}\thmnote{ \rm #3}}
\theoremstyle{bozont}    
\ifnum \value{are-there-sections}=0 {%
  \newtheorem{proclaim}{Theorem}
} 
\else {%
  \newtheorem{proclaim}{Theorem}[section]
} 
\fi
\newtheorem{thm}[proclaim]{Theorem}
\newtheorem{mainthm}[proclaim]{Main Theorem}
\newtheorem{cor}[proclaim]{Corollary} 
\newtheorem{cors}[proclaim]{Corollaries} 
\newtheorem{lem}[proclaim]{Lemma} 
\newtheorem{prop}[proclaim]{Proposition} 
\newtheorem{conj}[proclaim]{Conjecture}
\newtheorem{subproclaim}[equation]{Theorem}
\newtheorem{subthm}[equation]{Theorem}
\newtheorem{subcor}[equation]{Corollary} 
\newtheorem{sublem}[equation]{Lemma} 
\newtheorem{subprop}[equation]{Proposition} 
\newtheorem{subconj}[equation]{Conjecture}
\theoremstyle{bozont-sc}
\newtheorem{proclaim-special}[proclaim]{\specialthmname}
\newenvironment{proclaimspecial}[1]
     {\def\specialthmname{#1}\begin{proclaim-special}}
     {\end{proclaim-special}}
\theoremstyle{bozont-remark}
\newtheorem{rem}[proclaim]{Remark}
\newtheorem{subrem}[equation]{Remark}
\newtheorem{notation}[proclaim]{Notation} 
\newtheorem{assume}[proclaim]{Assumptions} 
\newtheorem{obs}[proclaim]{Observation} 
\newtheorem{example}[proclaim]{Example} 
\newtheorem{examples}[proclaim]{Examples} 
\newtheorem{complem}[equation]{Complement}
\newtheorem{const}[proclaim]{Construction}   
\newtheorem{ex}[proclaim]{Exercise} 
\newtheorem{subnotation}[equation]{Notation} 
\newtheorem{subassume}[equation]{Assumptions} 
\newtheorem{subobs}[equation]{Observation} 
\newtheorem{subexample}[equation]{Example} 
\newtheorem{subex}[equation]{Exercise} 
\newtheorem{claim}[proclaim]{Claim} 
\newtheorem{inclaim}[equation]{Claim} 
\newtheorem{subclaim}[equation]{Claim} 
\newtheorem{case}{Case} 
\newtheorem{subcase}{Subcase}   
\newtheorem{step}{Step}
\newtheorem{approach}{Approach}
\newtheorem{Fact}[proclaim]{Fact}
\newtheorem{fact}{Fact}
\newtheorem{subsay}{}
\newtheorem*{SubHeading*}{\SubHeadingName}%
\newtheorem{SubHeading}[proclaim]{\SubHeadingName}
\newtheorem{sSubHeading}[equation]{\sSubHeadingName}
\newenvironment{demo}[1] {\def\SubHeadingName{#1}\begin{SubHeading}}
  {\end{SubHeading}}%
\newenvironment{subdemo}[1]{\def\sSubHeadingName{#1}\begin{sSubHeading}}
  {\end{sSubHeading}} %
\newenvironment{demo-r}[1]{\def\SubHeadingName{#1}\begin{SubHeading-r}}
  {\end{SubHeading-r}}%
\newenvironment{subdemo-r}[1]{\def\sSubHeadingName{#1}\begin{sSubHeading-r}}
  {\end{sSubHeading-r}} %
\newenvironment{demo*}[1]{\def\SubHeadingName{#1}\begin{SubHeading*}}
  {\end{SubHeading*}}%
\newtheorem{defini}[proclaim]{Definition}
\newtheorem{question}[proclaim]{Question}
\newtheorem{subquestion}[equation]{Question}
\newtheorem{crit}[proclaim]{Criterion}
\newtheorem{pitfall}[proclaim]{Pitfall}
\newtheorem{addition}[proclaim]{Addition}
\newtheorem{principle}[proclaim]{Principle} 
\newtheorem{condition}[proclaim]{Condition}
\newtheorem{say}[proclaim]{}
\newtheorem{exmp}[proclaim]{Example}
\newtheorem{hint}[proclaim]{Hint}
\newtheorem{exrc}[proclaim]{Exercise}
\newtheorem{prob}[proclaim]{Problem}
\newtheorem{ques}[proclaim]{Question}    
\newtheorem{alg}[proclaim]{Algorithm}
\newtheorem{remk}[proclaim]{Remark}          
\newtheorem{note}[proclaim]{Note}            
\newtheorem{summ}[proclaim]{Summary}         
\newtheorem{notationk}[proclaim]{Notation}   
\newtheorem{warning}[proclaim]{Warning}  
\newtheorem{defn-thm}[proclaim]{Definition--Theorem}  
\newtheorem{convention}[proclaim]{Convention}  
\newtheorem*{ack}{Acknowledgment}
\newtheorem*{acks}{Acknowledgments}
\theoremstyle{bozont-def}    
\newtheorem{defn}[proclaim]{Definition}
\newtheorem{subdefn}[equation]{Definition}
\theoremstyle{bozont-reverse}    
\newtheorem{corr}[proclaim]{Corollary} 
\newtheorem{lemr}[proclaim]{Lemma} 
\newtheorem{propr}[proclaim]{Proposition} 
\newtheorem{conjr}[proclaim]{Conjecture}
\theoremstyle{bozont-reverse-sc}
\newtheorem{proclaimr-special}[proclaim]{\specialthmname}
\newenvironment{proclaimspecialr}[1]%
{\def\specialthmname{#1}\begin{proclaimr-special}}%
{\end{proclaimr-special}}
\theoremstyle{bozont-remark-reverse}
\newtheorem{remr}[proclaim]{Remark}
\newtheorem{subremr}[equation]{Remark}
\newtheorem{notationr}[proclaim]{Notation} 
\newtheorem{assumer}[proclaim]{Assumptions} 
\newtheorem{obsr}[proclaim]{Observation} 
\newtheorem{exampler}[proclaim]{Example} 
\newtheorem{exr}[proclaim]{Exercise} 
\newtheorem{claimr}[proclaim]{Claim} 
\newtheorem{inclaimr}[equation]{Claim} 
\newtheorem{SubHeading-r}[proclaim]{\SubHeadingName}
\newtheorem{sSubHeading-r}[equation]{\sSubHeadingName}
\newtheorem{SubHeadingr}[proclaim]{\SubHeadingName}
\newtheorem{sSubHeadingr}[equation]{\sSubHeadingName}
\newenvironment{demor}[1]{\def\SubHeadingName{#1}\begin{SubHeadingr}}{\end{SubHeadingr}}
\newtheorem{definir}[proclaim]{Definition}
\theoremstyle{bozont-def-newnum-reverse}    
\newtheorem{newnumr}[proclaim]{}
\theoremstyle{bozont-def-reverse}    
\newtheorem{defnr}[proclaim]{Definition}
\newtheorem{questionr}[proclaim]{Question}
\newtheorem{newnumspecial}[proclaim]{\specialnewnumname}
\newenvironment{newnum}[1]{\def\specialnewnumname{#1}\begin{newnumspecial}}{\end{newnumspecial}}
\numberwithin{equation}{proclaim}
\numberwithin{figure}{section} 
\newcommand\equinsect{\numberwithin{equation}{section}}
\newcommand\equinthm{\numberwithin{equation}{proclaim}}
\newcommand\figinthm{\numberwithin{figure}{proclaim}}
\newcommand\figinsect{\numberwithin{figure}{section}}
\newenvironment{sequation}{%
\numberwithin{equation}{section}%
\begin{equation}%
}{%
\end{equation}%
\numberwithin{equation}{proclaim}%
\addtocounter{proclaim}{1}%
}
\newcommand{\num}{\arabic{section}.\arabic{proclaim}}
\newenvironment{pf}{\smallskip \noindent {\sc Proof. }}{\qed\smallskip}
\newenvironment{enumerate-p}{
  \begin{enumerate}}
  {\setcounter{equation}{\value{enumi}}\end{enumerate}}
\newenvironment{enumerate-cont}{
  \begin{enumerate}
    {\setcounter{enumi}{\value{equation}}}}
  {\setcounter{equation}{\value{enumi}}
  \end{enumerate}}
\let\lenumi\labelenumi
\newcommand{\rmlabels}{\renewcommand{\labelenumi}{\rm \lenumi}}
\newcommand{\rmlabelsoff}{\renewcommand{\labelenumi}{\lenumi}}
\newenvironment{heading}{\begin{center} \sc}{\end{center}}
\newcommand\subheading[1]{\smallskip\noindent{{\bf #1.}\ }}
\newlength{\swidth}
\setlength{\swidth}{\textwidth}
\addtolength{\swidth}{-,5\parindent}
\newenvironment{narrow}{
  \medskip\noindent\hfill\begin{minipage}{\swidth}}
  {\end{minipage}\medskip}
\newcommand\nospace{\hskip-.45ex}
\makeatother

\title{The splitting principle and singularities}
\author{\me}
\date{\today}
\thanks{\mythanks}
\address{\myaddress}
\email{\myemail}
\urladdr{\myurladdr}
%

\begin{abstract}
  The splitting principle states that morphisms in a derived category do not
  ``split'' accidentally. This has been successsfully applied in several
  characterizations of rational, DB, and other singularities. In this article I prove
  a general statement, Theorem~\ref{thm:from-inj-surj-to-isom}, that implies many of
  the previous individual statements and improves some of the characterizations in
  the process. See Theorem~\ref{thm:kollar-kovacs-improve} for the actual statement.
\end{abstract}

\maketitle
\newcommand{\szabores}{Szab\'o-resolution\xspace}
\newcommand{\pairD}{\Delta}

\section{Introduction}

\noindent
The main guiding force of this article is the following principle.

\begin{demo*}{\bf {The Splitting Principle}}
  \it Morphisms in a derived category do not split accidentally.
\end{demo*}

\noindent
I will recall several theorems that justify this principle and make it precise in
their own context. For the necessary definitions please see the end of the
introduction.

\begin{rem}
  It is customary to casually use the word ``splitting'' to explain the statements of
  the theorems that follow.  However, the reader should be warned that one has to be
  careful with the meaning of this, because these ``splittings'' take place in a
  derived category, and derived categories are not abelian. For this reason, in the
  statements of the theorems below I use the terminology that a morphism admits a
  \emph{left inverse}.  In an abelian category this condition is equivalent to
  ``splitting'' and being a direct summand. With a slight abuse of language I labeled
  these as ``Splitting theorems'' cf. \eqref{thm:rtl-crit}, \eqref{thm:db-crit} and
  \eqref{thm:db-criterion}.
\end{rem}

\noindent
The first splitting theorem is a criterion for a singularity to be rational.

\begin{thm}[\protect{\cite[Theorem~1]{Kovacs00b}} ({\sc Splitting theorem
    I})]\label{thm:rtl-crit} 
  Let $\phi:Y\to X$ be a proper morphism of varieties over $\bC$ and
  $\varrho:\sO_X\to\rpforward\phi\sO_Y$ the associated natural morphism.  Assume
  that $Y$ has rational singularities and $\varrho$ has a left inverse, i.e., there
  exists a morphism (in the derived category of $\sO_X$-modules)
  $\varrho':\rpforward\phi\sO_Y\to\sO_X$ such that $\varrho'\circ\varrho$ is a
  quasi-isomorphism of $\sO_X$ with itself.
  Then $X$ has only rational singularities.
\end{thm}

\begin{rem}
  Note that $\phi$ in the theorem does not have to be birational or even generically
  finite.  It follows from the conditions that it is surjective.
\end{rem}

\begin{cor}\label{cor:split-rtl}
  Let $X$ be a complex variety and $\phi:Y\to X$ a resolution of singularities. If
  $\sO_X\to\rpforward\phi\sO_Y$ has a left inverse, then $X$ has rational
  singularities.
\end{cor}

\begin{cor}\label{cor:finite-rtl}
  Let $X$ be a complex variety and $\phi:Y\to X$ a finite morphism. If $Y$ has
  rational singularities, then so does $X$.
\end{cor}

Using this criterion it is quite easy to prove that log terminal singularities are
rational \cite[Theorem~4]{Kovacs00b}.  For related statements see \cite[5.22]{KM98} and the
references therein.

The next several splitting theorems concern DB singularities:

\begin{thm}[\protect{\cite[2.3]{Kovacs99} ({\sc Splitting theorem
      II})}]\label{thm:db-crit} 
  Let $X$ be a complex variety. If $\sO_X\to \DuBois X$ has a left inverse, then $X$
  has DB singularities.
\end{thm}

This criterion has several important consequences. It implies directly that rational
singularities are DB and it was used in \cite{KK10} to prove that log canonical
singularities are DB as well. In fact it is used in the proof of the next splitting
theorem.

\begin{thm}[\protect{\cite[1.6]{KK10} ({\sc Splitting theorem III})}]
  \label{thm:db-criterion}
  Let ${\phi}: Y\to X$ be a proper 
  morphism between reduced schemes of finite type over $\bC$.
  Let $W\subseteq X$ be a closed reduced subscheme with ideal sheaf $\sI_{W\subseteq
    X}$ and $F\leteq {\phi}^{-1}(W)\subset Y$ with ideal sheaf $\sI_{F\subseteq Y}$.
  Assume that the natural map $\varrho$
  $$
  \xymatrix{ \sI_{W\subseteq X} \ar[r]_-\varrho & \myR{\phi}_*\sI_{F\subseteq Y}
    \ar@{-->}@/_1.5pc/[l]_{\varrho'} }
  $$
  admits a left inverse $\varrho'$, that is,
  $\varrho'\circ\varrho=\id_{\sI_{W\subseteq X}}$. Then if $Y,F$, and $W$ all have DB
  singularities, then so does $X$.
\end{thm}

This criterion forms the cornerstone of the proof of the following theorem:

\begin{thm}[\protect{\cite[1.5]{KK10}}]\label{thm:main}
  Let $\phi:{{Y}}\to {{X}}$ be a proper surjective morphism with connected fibers
  between normal varieties. Assume that 
  $Y$ 
  has log canonical singularities and $K_{{Y}}\sim_{\bQ,\phi} 0$, that is, $K_Y$ is a
  $\phi$-relatively numerically trivial $\bQ$-divisor.  Then $X$ is DB.
\end{thm}

\begin{cor}[\protect{\cite[1.4]{KK10}}]
  Log canonical singularities are DB.
\end{cor}

\noindent
For the proofs and more general statements, please see \cite{KK10}.

\begin{subrem}
  Notice that in \eqref{thm:db-criterion} it is not required that ${\phi}$ be
  birational.  On the other hand the assumptions of the theorem and
  \cite[Thm~1]{Kovacs00b} imply that if $Y\setminus F$ has rational singularities,
  e.g., if $Y$ is smooth, then $X\setminus W$ has rational singularities as well.
\end{subrem}

This theorem is used in \cite{KK10} to derive various consequences, some of which
regard stable families and have strong consequences for moduli spaces of canonically
polarized varieties.  The interested reader should look at the original article to
obtain the full picture.

Finally, the newest splitting theorem is a generalization of \eqref{thm:db-crit} to
the case of pairs:

\begin{thm}[\protect{\cite[5.4]{Kovacs10a}} ({\sc Splitting theorem
    IV})]\label{thm:db-pairs-intro}
  Let $(X,\Sigma)$ be a reduced generalized pair. Assume that the natural morphism
  $\sI_{\Sigma\subseteq X}\to\Om_{X,\Sigma}^0$ has a left inverse. Then $(X,\Sigma)$
  is a DB pair.
\end{thm}

The main goal of this article is to prove a general splitting theorem that provides a
unified proof of \eqref{thm:db-crit}, \eqref{thm:db-criterion}, and
\eqref{thm:db-pairs-intro}. 
For the special definitions see \S\ref{sec:splitting-principle}.

\begin{thm}[({\sc The Splitting Principle})]\label{cor:splitting-principle-intro}
  Let $\sch=\sch_k$ be the category of schemes of finite type over a fixed
  algebraically closed field $k$, $\sF$ and $\sG:\sch\to D\sch$ be two consistent
  ordinary functors, and $\eta:\sF\to\sG$ a consistent cohomologically surjective
  natural transformation as defined in \eqref{def:consistent-functors} and
  \eqref{def:cohomologically-surjective}.  Let $Y$ be a generically reduced
  quasi-projective scheme of finite type over $k$ and $V\subseteq Y$ a dense open
  subset such that $\eta_V:\sF(V)\isom \sG(V)$ is a quasi-isomorphism. Assume that
  $\eta_Y:\sF(Y)\to \sG(Y)$ has a left inverse.
  Then it 
  is a quasi-isomorphism.
\end{thm}

As a corollary of this theorem we obtain a more general statement that does not only
imply these three theorems, but it also strengthens \cite[1.6]{KK10} (see
\eqref{thm:db-criterion}) by changing a simple one way implication to an equivalence.
For the precise statement please see \eqref{thm:kollar-kovacs-improve}. It may also
be of interest that this constitutes a new proof of \cite[1.6]{KK10} that is
considerably simpler than the original one.

Finally, let me address the point that the reader have probably noticed. I have
listed four splitting theorems and the abstract theorem proved in this article
implies three of them. Considering the nature of the four theorems this is not
surprising, but the abstract theorem \eqref{thm:db-pairs-intro} may actually be used
to derive criteria similar to \eqref{thm:rtl-crit} that implies that certain
singularities are rational. I will leave figuring out these possibilities for the
reader. I would also like to issue a challenge to generalize
\eqref{thm:db-pairs-intro} to a statement that implies \ref{thm:rtl-crit} as well.

\subsection{\bf Definitions and Notation}\label{demo:defs-and-not}
%
  If $\phi:Y\to Z$ is a birational morphism, then $\exc(\phi)$ will denote the
  \emph{exceptional set} of $\phi$. For a closed subscheme $W\subseteq X$, the ideal
  sheaf of $W$ is denoted by $\sI_{W\subseteq X}$ or if no confusion is likely, then
  simply by $\sI_W$.  For a point $x\in X$, $\kappa(x)$ denotes the residue field of
  $\sO_{X,x}$.
  
  For morphisms $\phi:X\to B$ and $\vartheta: T\to B$, the symbol $X_T$ will denote
  $X\times_B T$ and $\phi_T:X_T\to T$ the induced morphism.  In particular, for $b\in
  B$ I write $X_b = \phi^{-1}(b)$.
  Of course, by symmetry, we also have the notation $\vartheta_X:T_X\simeq X_T\to X$
  and if $\sF$ is an $\sO_X$-module, then $\sF_T$ will denote the $\sO_{X_T}$-module
  $\vartheta_X^*\sF$.




  Let $X$ be a scheme. 
  Let $D_{\rm filt}(X)$ denote the derived category of filtered complexes of
  $\sO_{X}$-modules with differentials of order $\leq 1$ and $D_{\rm filt, coh}(X)$
  the subcategory of $D_{\rm filt}(X)$ of complexes $\cx K$, such that for all $i$,
  the cohomology sheaves of $Gr^{i}_{\rm filt}K^{\kdot}$ are coherent cf.\
  \cite{DuBois81}, \cite{GNPP88}.  Let $D(X)$ and $D_{\rm coh}(X)$ denote the derived
  categories with the same definition except that the complexes are assumed to have
  the trivial filtration.  The superscripts $+, -, b$ carry the usual meaning
  (bounded below, bounded above, bounded).  Isomorphism in these categories is
  denoted by $\qis$.  A sheaf $\sF$ is also considered as a complex $\sF$ with
  $\sF^0=\sF$ and $\sF^i=0$ for $i\neq 0$.  If $\sfA$ is a complex in any of the
  above categories, then $h^i(\sfA)$ denotes the $i$-th cohomology sheaf of
  $\sfA$.
  The \emph{support} of $\sfA$ is the union of the supports of its cohomology
  sheaves: $\supp\sfA\leteq \bigcup_i \supp h^i(\sfA)$.

  The right derived functor of an additive functor $F$, if it exists, is denoted by
  $\myR F$ and $\myR^iF$ is short for $h^i\circ \myR F$. Furthermore, $\bH^i$,
  $\bH^i_{\rm c}$, $\bH^i_Z$ , and $\sH^i_Z$ will denote $\myR^i\Gamma$,
  $\myR^i\Gamma_{\rm c}$, $\myR^i\Gamma_Z$, and $\myR^i\sH_Z$ respectively, where
  $\Gamma$ is the functor of global sections, $\Gamma_{\rm c}$ is the functor of
  global sections with proper support, $\Gamma_Z$ is the functor of global sections
  with support in the closed subset $Z$, and $\sH_Z$ is the functor of the sheaf of
  local sections with support in the closed subset $Z$.  Note that according to this
  terminology, if $\phi\col Y\to X$ is a morphism and $\sF$ is a coherent sheaf on
  $Y$, then $\myR\phi_*\sF$ is the complex whose cohomology sheaves give rise to the
  usual higher direct images of $\sF$.

  I will often use the notion that a morphism ${f}: \sfA\to \sfB$ in a derived
  category \emph{has a left inverse}. This means that there exists a morphism
  $f^\ell: \sfB\to \sfA$ in the same derived category such that
  $f^\ell\circ{f}:\sfA\to\sfA$ is the identity morphism of $\sfA$. I.e., $f^\ell$ is
  a \emph{left inverse} of ${f}$.

  I will also make the following simplification in notation. First observe that if
  $\iota:\Sigma \into X$ is a closed embedding of schemes then $\iota_*$ is exact and
  hence $\myR\iota_*=\iota_*$. This allows one to make the following harmless abuse
  of notation: If $\sfA\in\ob D(\Sigma)$, then, as usual for sheaves, I will drop
  $\iota_*$ from the notation of the object $\iota_*\sfA$. In other words, I will,
  without further warning, consider $\sfA$ an object in $D(X)$.

  A \emph{generalized pair} $(X,\Sigma)$ consists of an equidimensional variety
  (i.e., a reduced scheme of finite type over a field $k$) $X$ and a closed subscheme
  $\Sigma\subseteq X$.  A morphism of generalized pairs $\phi:(Y,\Gamma)\to
  (X,\Sigma)$ is a morphism $\phi:Y\to X$ such that $\phi(\Gamma)\subseteq \Sigma$.
  A \emph{reduced generalized pair} is a generalized pair $(X,\Sigma)$ such that
  $\Sigma$ is reduced.

  The \emph{log resolution} of a generalized pair $(X, W)$ is a proper birational
  morphism $\pi: Y\to X$ such that $\exc(\pi)$ is a divisor and
  $\pi^{-1}W+\exc(\pi)$ is an snc divisor.

  Let $X$ be a complex scheme and $\Sigma$ a closed subscheme whose complement in $X$
  is dense. Then $(X_{\kdot}, \Sigma_\kdot)\to (X, \Sigma)$ is a \emph{good
    hyperresolution} if $X_\kdot\to X$ is a hyperresolution, and if
  $U_\kdot=X_\kdot\times_X (X\setminus \Sigma)$ and $\Sigma_\kdot=X_\kdot\setminus
  U_\kdot$, then for all $\alpha$ either $\Sigma_\alpha$ is a divisor with normal
  crossings on $X_\alpha$ or $\Sigma_\alpha=X_\alpha$. Notice that it is possible
  that $X_\kdot$ has some components that map into $\Sigma$. These components are
  contained in $\Sigma_\mydot$.  For more details and the existence of such
  hyperresolutions see \cite[6.2]{DuBois81} and \cite[IV.1.21, IV.1.25,
  IV.2.1]{GNPP88}.  For a primer on hyperresolutions see the appendix of
  \cite{Kovacs-Schwede11}.

  Let $(X,\Sigma)$ be a reduced generalized pair and let $\Om^\kdot_{X,\Sigma}$
  denote the Deligne-Du~Bois complex of $(X,\Sigma)$. The $0^\text{th}$ associated
  graded quotient of this will be denoted by $\Om^0_{X,\Sigma}$.  If
  $\Sigma=\emptyset$, it will be dropped from the notation: $\Om^0_{X}\leteq
  \Om^0_{X,\Sigma}$. For more details see Steenbrink \cite[\S 3]{Steenbrink85} and
  \cite[3.9]{Kovacs10a} and the relevant references in the latter article.



\section{The Abstract Splitting Principle}\label{sec:splitting-principle}


In this section I will introduce a few new notions to generalize the conditions
needed to prove the desired abstract theorem and prove a few general statements
leading to the main theorem.

First we need a definition mainly for simplifying notation and terminology.

\begin{defini}
  Let $\sch$ be a category of schemes and $D\sch$ the following associated category
  of pairs: An object of $D\sch$ is a pair $(X,\sfA)$ consisting of a scheme
  $X\in\ob\sch$ and an object $\sfA\in\Ob D(X)$; and a morphism $\phi:(X,\sfA)\to
  (Y,\sfB)$ consist of a morphism of schemes $\phi:X\to Y$ (denoted by the same
  symbol unless confusion is possible) and a morphism in $D(Y)$, $\phi^\#:\sfB\to
  \myR\phi_*\sfA$. Observe that there exists a natural embedding of $\sch$ into
  $D\sch$ by mapping any $X\in \Ob\sch$ to the pair $(X,\sO_X)\in \Ob D\sch$. Note
  that $\sch$ is \emph{a} category of schemes, not necessarily \emph{the} category of
  schemes. In particular, especially in applications, we will often assume that $\sch$
  is the category of schemes of finite type over an algebraically closed field, for
  instance $\mathbb C$.

  A functor $\sS:\sch\to D\sch$ will be called \emph{ordinary} if $\sS(X)=(X,\sF(X))$
  for any $X\in\Ob \sch$, i.e., the scheme part of the pair $\sS(X)$ is equal to the
  original scheme $X$.  In this case we will identify $\sS=(\id_\sch,\sF)$ with
  $\sF$.
\end{defini}

Next we consider a condition that can be reasonably expected from any geometrically
defined functors.

\begin{defini}\label{def:consistent-functors}
  Let $\sch=\sch_k$ be the category of schemes of finite type over a fixed
  algebraically closed field $k$ and $\sF:\sch\to D\sch$ an ordinary functor.  $\sF$
  will be called a \emph{consistent functor} if for any quasi-projective generically
  reduced scheme $X\in\ob\sch$ and any general hyperplane section $H\subseteq X$
  there exist a natural isomorphism
  $$\sF(X)\otimes_{\myL} \sO_H\qis \sF(H).$$

  If $\sF$ and $\sG:\sch\to D\sch$ are two consistent (ordinary) functors, then a
  natural transformation $\eta:\sF\to\sG$ is called a \emph{consistent natural
    transformation} if for any quasi-projective generically reduced scheme
  $X\in\ob\sch$, and any general hyperplane section $H\subseteq X$ there exists a
  commutative diagram:
  $$
  \xymatrix{%
    \sF(X)\otimes_{\myL} \sO_H \ar[r]^-{\qis}
    \ar[d]_{\eta_X\otimes_{\myL}\id_{\sO_H}} & \sF(H)
    \ar[d]^{\eta_H} \\
    \sG(X)\otimes_{\myL} \sO_H \ar[r]^-{\qis} & \sG(H). }
  $$
\end{defini}

The next definition is an abstract way to grasp a condition implied by the Hodge
decomposition of singular cohomology that plays a key role in the proof of
\eqref{thm:db-crit}. The fact that the Hodge-to-de Rham spectral sequence for a
smooth complex projective variety $X$ degenerates at $E_1$ implies that the natural
map on cohomology
$$
H^i(X,\bC)\onto H^i(X,\sO_X)
$$
is surjective for all $i$. For not necessarily smooth projective schemes the target
of the equivalent of this surjectivity is the corresponding hypercohomology of
$\Om^0_X$. For our purposes this implies that for an arbitrary complex projective
scheme of finite type there exists a natural map 
$$
H^i(X,\sO_X)\onto \bH^i(X,\Om^0_X)
$$
which is surjective for all $i$. This surjectivity comes from singular cohomology and
Hodge theory, but once we have it in this form the rest of the proof of
\eqref{thm:db-crit} does not require either one of those, in particular, it does not
require us to work over the complex numbers or even in characteristic zero (except
for the definition of $\Om^0_X$).

\begin{defini}\label{def:cohomologically-surjective}
  Let $\sch=\sch_k$ be the category of schemes of finite type over a fixed
  algebraically closed field $k$, $\sF$ and $\sG:\sch\to D\sch$ be two consistent
  ordinary functors, and $\eta:\sF\to\sG$ a consistent natural transformation. 

  Then $\eta$ will be called \emph{cohomologically surjective} if for any
  generically reduced %
  affine %
  scheme $X\in\ob\sch$ there exists an $\ol X\in \ob\sch$ such that $X\subseteq \ol
  X$ is an open set and
  $$\bH^i(\eta_{\ol X}):\bH^i(\ol X,\sF(\ol X))\onto \bH^i(\ol X,\sG(\ol X))$$ is
  surjective for all $i$.
\end{defini}

The following is a key ingredient of the overall argument. The main point of this
statement is to relay the surjectivity obtained for projective schemes to
quasi-projective ones. In order to avoid losing important information this is done by
using local cohomology.

\begin{thm}
  \label{thm:meta-surjectivity}
  Let $\sF,\sG:\sch\to D\sch$ be two ordinary functors and $\eta:\sF\to\sG$ a natural
  transformation.  Further let $\ol X$ be a scheme, $X\subseteq \ol X$ an open
  subscheme, and $P\subset\ol X$ a closed subscheme. Assume that $P\subseteq X$ and
  let $U\leteq X\setminus P$. Further assume that
  \begin{enumerate-p}
  \item $\bH^i(\eta_{\ol X}):\bH^i(\ol X,\sF(\ol X))\onto \bH^i(\ol X,\sG(\ol X))$ is
    surjective for all $i$, and
  \item $\bH^i(\eta_{U}):\bH^i(U, \sF(U))\isom \bH^i(U,\sG(U))$ is an isomorphism for
    all $i$.
  \end{enumerate-p}
  Then $\bH^i_P(\eta_{X}):\bH^i_P(X,\sF(X))\onto \bH^i_P(X,\sG(X))$ is surjective for
  all $i$.
\end{thm}

\begin{proof}
  Let $Q=\ol X\setminus X$, $Z=P\disjoint Q$, and $U=\ol X\setminus Z=X\setminus P$.
  Consider the exact triangle of functors,
    \begin{equation}
    \label{eq:13}
    \xymatrix{%
      \bH^0_Z({\ol X},\blank ) \ar[r] & \bH^0({\ol X},\blank ) \ar[r] & \bH^0(U,\blank )
      \ar[r]^-{+1} & 
    }
  \end{equation}
  and apply it to the morphism $\eta_{\ol X}:\sF(\ol X)\to\sG(\ol X)$. One obtains a
  morphism of two long exact sequences:
  $$
  \hskip-4.75em\xymatrix{%
    \dots \ar[r] & \bH^{i-1}(U,{\sF(U)} ) \ar[d]^{\alpha_{i-1}}\ar[r] &
    \bH^{i}_Z({\ol X},{\sF(\ol X)} ) \ar[d]^{\beta_i}\ar[r] & \bH^{i}({\ol
      X},{\sF(\ol X)} ) \ar[d]^{\gamma_i}\ar[r] &
    \bH^{i}(U,{\sF(U)} ) \ar[d]^{\alpha_i} \ar[r] & \dots  \\
    \dots \ar[r] & \bH^{i-1}(U,\sG(U) ) \ar[r] & \bH^{i}_Z({\ol X},\sG(\ol X) )
    \ar[r] & \bH^{i}({\ol X},\sG(\ol X) ) \ar[r] & \bH^{i}(U,\sG(U) ) \ar[r] & \dots
    .}
  $$
  By assumption, $\alpha_i$ is an isomorphism and $\gamma_i$ is surjective for all
  $i$. Then by the 5-lemma, $\beta_i$ is also surjective for all $i$.
    
  By construction $P\cap Q=\emptyset$ and hence
  \begin{align*}
    \bH^i_Z({\ol X}, \sF(\ol X)) &\simeq \bH^i_P({\ol X}, \sF(\ol X)) \oplus
    \bH^i_Q({\ol X}, \sF(\ol X)) \\
    \bH^i_Z({\ol X}, \sG(\ol X)) &\simeq \bH^i_P({\ol X}, \sG(\ol X)) \oplus
    \bH^i_Q({\ol X}, \sG(\ol X))
  \end{align*}
  It follows that the natural map (which is also the restriction of $\beta_i$),
  $$
  \bH^i_P(\ol X, \sF(\ol X)) \to \bH^i_P(\ol X, \sG(\ol X))
  $$
  is surjective for all $i$.  Now, by excision on local cohomology one has that
  $$
  \bH^i_P(\ol X, \sF(\ol X)) \simeq \bH^i_P(X, \sF(X) \quad\text{and}\quad
  \bH^i_P(\ol X, \sG(\ol X)) \simeq \bH^i_P(X, \sG(X)),
  $$
  and so the desired statement follows.
\end{proof}

The next theorem is the main result of this article. It generalizes the statement and
proof of those theorems mentioned in the introduction to a quite general level. In
the next section I will explain how this implies almost immediately those three
results and stregthens one of them. However, it seems reasonable to expect that this
form will be used later to prove similar statements in different situations.

\begin{thm}\label{thm:from-inj-surj-to-isom}
  Let $\sch=\sch_k$ be the category of schemes of finite type over a fixed
  algebraically closed field $k$, $\sF$ and $\sG:\sch\to D\sch$ be two consistent
  ordinary functors, and $\eta:\sF\to\sG$ a consistent cohomologically surjective
  natural transformation as defined in \eqref{def:consistent-functors} and
  \eqref{def:cohomologically-surjective}.  Let $Y$ be a generically reduced
  quasi-projective scheme of finite type over $k$ and $V\subseteq Y$ a dense open
  subset such that $\eta_V:\sF(V)\isom \sG(V)$ is a quasi-isomorphism. If for any
  general complete intersection $X\subseteq Y$ and any closed subscheme $Z\subseteq
  X\setminus V$,
  \begin{equation}
    \label{eq:inj}
    \bH^i_Z(\eta_{X}):\bH^i_Z(X,\sF(X))\into \bH^i_Z(X,\sG(X))
  \end{equation}
  is injective for all $i$, then $\eta_Y:\sF(Y)\isom\sG(Y)$ is a quasi-isomorphism.
\end{thm}

\begin{proof}
  Let $\sD(Y)$ be an object in $D\sch$ that completes the morphism
  $\eta_Y:\sF(Y)\to\sG(Y)$ to a distinguished triangle:
  $$
  \xymatrix{%
    \sF(Y)\ar[r] & \sG(Y) \ar[r] & \sD(Y) \ar[r]^-{+1} & . 
  }
  $$

  Let $T=\supp\sD(Y)\subset Y\setminus V$, a closed subset of $Y$.  We need to prove
  that $\sD(Y)\qis 0$, that is, that $T=\emptyset$. Suppose that $T\neq\emptyset$ and
  we will derive a contradiction.

  By assumption $Y\setminus T\supset V$ a dense open subset of $Y$. It follows that
  if $X\subseteq Y$ is a general complete intersection of $Y$ of the appropriate
  codimension, then $P\leteq X\cap T$ is a finite closed \emph{non-empty} subset.

  Since $\eta$ is consistent, cf.\ \eqref{def:consistent-functors}, setting
  $\sD(X)\leteq \sD(Y)\otimes_{\myL}\sO_X$ one obtains a distinguished triangle:
  $$
  \xymatrix{%
    \sF(X)\ar[r] & \sG(X) \ar[r] & \sD(X) \ar[r]^-{+1} & ,
  }
  $$
  such that $P=\supp\sD(X)$. We will prove that $P=\emptyset$ which is a
  contradiction to the way $P$ was defined.

  As $P$ is finite we may assume that $X$ is affine.  Consider $X\subseteq\ol X$
  given by the fact that $\eta$ is cohomologically surjective, cf.\
  \eqref{def:cohomologically-surjective}. Again, since $P$ is a finite set, it
  follows that $P\subset \ol X$ is also closed and then it follows by %
  \eqref{thm:meta-surjectivity} and %
  the assumption in \eqref{eq:inj} %
  that
  \begin{equation}
    \label{eq:bb2}
    \bH^i_P(\eta_{X}):\bH^i_P(X,\sF(X))\isom \bH^i_P(X,\sG(X))
  \end{equation}
  is an isomorphism for all $i$, and then it follows that
  \begin{equation*}
    \bH^i_P(X,\sD(X))=0\quad\text{for all $i$.}
  \end{equation*}
  Since $\supp\sD(X)= {P}$ it also follows that
  $$\bH^i(X\setminus P, \sD(X))=0$$ for all $i$ as well, and then 
  \begin{equation}
    \label{eq:H-is-zero}
    \bH^i(X, \sD(X))=0  
  \end{equation}
  for all $i$ by the long exact sequence induced by (\ref{eq:13}) applied
  with $\ol X\leftrightarrow X$ and $Z\leftrightarrow P$.

  Since $X$ is affine, %
  the spectral sequence that computes hypercohomology from the cohomology of the
  cohomology sheaves of the complex $\sD(X)$ degenerates and gives that $\bH^i(X,
  \sD(X))=H^0(X, h^i(\sD(X)))$ for all $i$. 
  It follows by (\ref{eq:H-is-zero}) that $h^i(\sD(X))=0$ for all $i$.  Therefore
  $\sD(X)\qis 0$ and hence $P=\emptyset$. We arrived to our promised contradiction,
  so the desired statement is proven.
\end{proof}

The following is a straightforward corollary of \eqref{thm:from-inj-surj-to-isom},
its main value is in that its conditions may be easier to verify.

\begin{cor}[({\sc The Splitting Principle})]\label{cor:splitting-principle}
  Let $\sch=\sch_k$ be the category of schemes of finite type over a fixed
  algebraically closed field $k$, $\sF$ and $\sG:\sch\to D\sch$ be two consistent
  ordinary functors, and $\eta:\sF\to\sG$ a consistent cohomologically surjective
  natural transformation as defined in \eqref{def:consistent-functors} and
  \eqref{def:cohomologically-surjective}.  Let $Y$ be a generically reduced
  quasi-projective scheme of finite type over $k$ and $V\subseteq Y$ a dense open
  subset such that $\eta_V:\sF(V)\isom \sG(V)$ is a quasi-isomorphism. Assume that
  $\eta_Y:\sF(Y)\to \sG(Y)$ has a left inverse.
  Then it 
  is a quasi-isomorphism.
\end{cor}

\begin{proof}
  If $\eta_Y:\sF(Y)\to \sG(Y)$ has a left inverse, then the same holds for $\eta_X$
  for any general complete intersection $X\subseteq Y$, and so
  $$\bH^i_Z(\eta_{X}):\bH^i_Z(X,\sF(X))\into \bH^i_Z(X,\sG(X))$$ is injective 
  for any closed subset $Z\subseteq X$. Then the statement follows from
  \eqref{thm:from-inj-surj-to-isom}.
\end{proof}

\section{Applications}

In this sections I show how \eqref{thm:from-inj-surj-to-isom} implies
\eqref{thm:db-crit}, \eqref{thm:db-criterion}, and \eqref{thm:db-pairs-intro}.

\subsection{DB singularities}

The first application is one of the first appearances of the splitting principle:

\begin{thm}[\protect{\cite[2.3]{Kovacs99} see
    \eqref{thm:db-crit}}]\label{thm:db-sings}
  Let $X$ be a scheme of finite type over $\mathbb C$. If the natural map
  $\sO_X\to\Om_X^0$ admits a left inverse, then $X$ has DB singularities.
\end{thm}

\begin{proof}
  Let $\sch=\sch_k$ be the category of schemes of finite type over $\mathbb C$,
  $\sF(\blank)=\sO_{\blank}$ and $\sG(\blank)=\Om^0_{\blank}$. These define two
  ordinary functors $\sch\to D\sch$. They are both consistent as defined in
  \eqref{def:consistent-functors} cf.\ \cite[2.6]{Kovacs10a} and there exists a
  consistent cohomologically surjective natural transformation $\eta:\sF\to\sG$ by
  \cite[4.5]{DuBois81}. Let $V\leteq X\setminus \Sing X$. Then $\eta_V:\sF(V)\isom
  \sG(V)$ is a quasi-isomorphism. Then the statement follows from
  \eqref{cor:splitting-principle}.
\end{proof}

\subsection{DB pairs}

\begin{thm}[\protect{\cite[5.4]{Kovacs10a} see
    \eqref{thm:db-pairs-intro}}]\label{thm:db-pairs}  
  Let $(X,\Sigma)$ be a reduced generalized pair. Assume that the natural morphism
  $\sI_{\Sigma\subseteq X}\to\Om_{X,\Sigma}^0$ has a left inverse. Then $(X,\Sigma)$
  is a DB pair.
\end{thm}

\begin{proof}
  Let $\sch=\sch_k$ be the category of subschemes $\Sigma$ of $X$ of finite type over
  $\mathbb C$, $\sF(\Sigma)=\sI_{\Sigma\subseteq X}$ and
  $\sG(\Sigma)=\Om^0_{X,\Sigma}$. These define two ordinary functors $\sch\to D\sch$.
  They are both consistent as defined in \eqref{def:consistent-functors} by
  \cite[3.18]{Kovacs10a} and there exists a consistent cohomologically surjective
  natural transformation $\eta:\sF\to\sG$ by \cite[4.2]{Kovacs10a}. Let $V\leteq
  (X\setminus \Sing X)\setminus\supp\Sigma$. Then $\eta_V:\sF(V)\isom \sG(V)$ is a
  quasi-isomorphism. Then the statement follows from \eqref{cor:splitting-principle}.
\end{proof}

\subsection{The Koll\'ar-Kov\'acs DB criterion}

\begin{thm}\label{thm:kollar-kovacs-improve}
  Let ${f}: Y\to X$ be a proper 
  morphism between reduced schemes of finite type over $\bC$, $W\subseteq X$ an
  arbitrary subscheme, and $F\leteq {f}^{-1}(W)$, equipped with the induced
  reduced subscheme structure. Assume that the natural map $\varrho$
  $$
  \xymatrix{ \sI_{W\subseteq X} \ar[r]_-\varrho & \myR{f}_*\sI_{F\subseteq Y}
    \ar@{-->}@/_1.5pc/[l]_{\varrho'} }
  $$
  admits a left inverse $\varrho'$.
  Then if $(Y,F)$ is a DB pair, then so is $(X,W)$. In particular, if $(Y,F)$ is a DB
  pair, then $X$ is DB if and only if $W$ is DB.
\end{thm}

\begin{proof}
  By functoriality one obtains a commutative diagram
  $$
  \xymatrix{%
    \sI_{W\subseteq X} \ar[r]^-\varrho \ar[d]_\alpha & \myR{f}_*\sI_{F\subseteq Y} \ar[d]_\gamma^{\qis} \\
    \Om^0_{W\subseteq X} \ar[r]_\beta & \myR{f}_*\Om^0_{Y,F}. }
  $$
  Since $(Y,F)$ is assumed to be a DB pair, it follows that $\gamma$ is a
  quasi-isomorphism and hence $\varrho'\circ\gamma^{-1}\circ \beta$ is a left inverse
  to $\alpha$. Then the statement follows by \eqref{thm:db-pairs}.
\end{proof}

\begin{cor}[\protect{\cite[1.6]{KK10} see
    \eqref{thm:db-criterion}}]\label{sec:kollar-kovacs-original}
    Let ${f}: Y\to X$ be a proper 
  morphism between reduced schemes of finite type over $\bC$, $W\subseteq X$ an
  arbitrary subscheme, and $F\leteq {f}^{-1}(W)$, equipped with the induced
  reduced subscheme structure. Assume that the natural map $\varrho$
  $$
  \xymatrix{ \sI_{W\subseteq X} \ar[r]_-\varrho & \myR{f}_*\sI_{F\subseteq Y}
    \ar@{-->}@/_1.5pc/[l]_{\varrho'} }
  $$
  admits a left inverse $\varrho'$, that is, $\rho'\circ\rho=\id_{\sI_{W\subseteq
      X}}$. Then if $Y,F$, and $W$ all have DB singularities, then so does $X$.
\end{cor}


\def\cprime{$'$} \def\polhk#1{\setbox0=\hbox{#1}{\ooalign{\hidewidth
  \lower1.5ex\hbox{`}\hidewidth\crcr\unhbox0}}} \def\cprime{$'$}
  \def\cprime{$'$} \def\cprime{$'$} \def\cprime{$'$}
  \def\polhk#1{\setbox0=\hbox{#1}{\ooalign{\hidewidth
  \lower1.5ex\hbox{`}\hidewidth\crcr\unhbox0}}} \def\cdprime{$''$}
  \def\cprime{$'$} \def\cprime{$'$} \def\cprime{$'$} \def\cprime{$'$}
\providecommand{\bysame}{\leavevmode\hbox to3em{\hrulefill}\thinspace}
\providecommand{\MR}{\relax\ifhmode\unskip\space\fi MR}
\providecommand{\MRhref}[2]{%
  \href{http://www.ams.org/mathscinet-getitem?mr=#1}{#2}
}
\providecommand{\href}[2]{#2}

\end{document}